\documentclass[12pt,a4paper,oneside]{amsart}

\usepackage{amsfonts, amsmath, amssymb, amsthm, amscd, hyperref, textcomp}
\usepackage{anysize}
\marginsize{2cm}{2cm}{2cm}{2cm}
\usepackage{graphicx}

\newtheorem{theorem}{Theorem}[section]
\newtheorem{lemma}[theorem]{Lemma}

\newtheorem{corollary}[theorem]{Corollary}

\theoremstyle{definition}

\theoremstyle{remark}
\newtheorem{remark}[theorem]{Remark}

\numberwithin{equation}{section}

\DeclareMathOperator{\inte}{int}
\DeclareMathOperator{\conv}{conv}

\begin{document}

\title{On the Circle Covering Theorem by A.~W.~Goodman and R.~E.~Goodman}

\author{Arseniy~Akopyan}
\address{A.~Akopyan,
		Institute of Science and Technology Austria (IST Austria), Am Campus~1, 3400 Klosterneuburg, Austria}
\email{akopjan@gmail.com}

\author{Alexey~Balitskiy}
\address{A.~Balitskiy,
		Moscow Institute of Physics and Technology, Institutskiy per. 9, Dolgoprudny, Russia 141700, \newline
		Institute for Information Transmission Problems RAS, Bolshoy Karetny per. 19, Moscow, Russia 127994, \newline	
		Dept. of Mathematics, Massachusetts Institute of Technology, 182 Memorial Dr., Cambridge, MA 02142}
		\email{alexey\_m39@mail.ru}

\author{Mikhail~Grigorev}
\address{M.~Grigorev,
		Moscow Institute of Physics and Technology, Institutskiy per. 9, Dolgoprudny, Russia 141700,}
		\email{mikhail.grigorev@phystech.edu}

\maketitle

\begin{abstract}
In 1945, A.~W.~Goodman and R.~E.~Goodman proved the following conjecture by P.~Erd\H{o}s: Given a family of (round) disks of radii $r_1$, $\ldots$, $r_n$ in the plane, it is always possible to cover them by a disk of radius $R = \sum r_i$, provided they cannot be separated into two subfamilies by a straight line disjoint from the disks. In this note we show that essentially the same idea may work for different analogues and generalizations of their result. In particular, we prove the following: Given a family of positive homothetic copies of a fixed convex body $K \subset \mathbb{R}^d$ with homothety coefficients $\tau_1, \ldots, \tau_n > 0$ it is always possible to cover them by a translate of $\frac{d+1}{2}\left(\sum \tau_i\right)K$, provided they cannot be separated into two subfamilies by a hyperplane disjoint from the homothets.

% \keywords{Goodman--Goodman theorem \and Non-separable family \and Positive homothets}
% \PACS{PACS code1 \and PACS code2 \and more}
 % \subclass{M52C10\ 52C17}
\end{abstract}

\section{Introduction}

Consider a family $\mathcal{K}$ of positive homothetic copies of a fixed convex body $K \subset \mathbb{R}^d$ with homothety coefficients $\tau_1, \ldots, \tau_n > 0$. Following Hadwiger~\cite{hadwiger1947nonseparable}, we call $\mathcal{K}$ \emph{non-separable} if any hyperplane $H$ intersecting $\conv \bigcup \mathcal{K}$ intersects a member of $\mathcal{K}$. Answering a question by Erd\H{o}s, A.~W.~Goodman and R.~E.~Goodman~\cite{goodman1945circle} proved the following assertion:

\begin{theorem}[A.~W.~Goodman, R.~E.~Goodman, 1945]
\label{thm:goodman}
Given a non-separable family $\mathcal{K}$ of Euclidean balls of radii $r_1, \ldots, r_n$ in $\mathbb{R}^d$, it is always possible to cover them by a ball of radius $R = \sum r_i$.
\end{theorem}

Let us outline here the idea of their proof since we are going to reuse it in different settings.

First, A.~W.~Goodman and R.~E.~Goodman prove the following lemma, resembling the $1$-dimensional case of the general theorem:

\begin{lemma}%[A.~W.~Goodman, R.~E.~Goodman, 1945]
\label{lem:segm}
Let $I_1, \ldots, I_n \subset \mathbb{R}$ be segments of lengths $\ell_1, \ldots, \ell_n$ with midpoints $c_1, \ldots, c_n$. Assume the union $\bigcup I_i$ is a segment (i.e. the family of segments is non-separable). Then the segment $I$ of length $\sum \ell_i$ with midpoint at the center of mass $c = \frac{\sum \ell_i c_i}{\sum \ell_i}$ covers $\bigcup I_i$.
\end{lemma}

Next, for a family $\mathcal{K} = \{o_i + r_i B\}$ ($B$ denotes the unit ball centered at the origin of $\mathbb{R}^d$), A.~W.~Goodman and R.~E.~Goodman consider the point $o = \frac{\sum r_i o_i}{\sum r_i}$ (i.e., the center of mass of~$\mathcal{K}$ if the weights of the balls are chosen to be proportional to the radii). They project the whole family onto $d$ orthogonal directions (chosen arbitrarily) and apply Lemma~\ref{lem:segm} to show that the ball of radius $R = \sum r_i$ centered at $o$ indeed covers $\mathcal{K}$.

In~\cite{bezdek2016non}, K.~Bezdek and Z.~L\'angi show that Theorem~\ref{thm:goodman} actually holds not only for balls but also for any centrally-symmetric bodies:

\begin{theorem}[K.~Bezdek and Z.~Langi, 2016]
\label{thm:symm}
Given a non-separable family of homothets of centrally-symmetric convex body $K \subset \mathbb{R}^d$ with homothety coefficients $\tau_1, \ldots, \tau_n > 0$ it is always possible to cover them by a translate of $\left(\sum \tau_i\right)K$.
\end{theorem}

The idea of their proof is to use Lemma~\ref{lem:segm} to deduce the statement for the case when $K$ is a hypercube, and then deduce the result for sections of the hypercube (which can approximate arbitrary centrally-symmetric bodies).

It is worth noticing that Theorem~\ref{thm:symm} follows from Lemma~\ref{lem:segm} by a more direct argument (however, missed by A.~W.~Goodman and R.~E.~Goodman). In 2001 F.~Petrov proposed a particular case of the problem (when $K$ is a Euclidean ball) to Open Mathematical Contest of Saint Petersburg Lyceum~\textnumero239 \cite{piter-olimpiadi2000-2002}. He assumed the following solution (working for any symmetric $K$ as well): For a family $\mathcal{K} = \{o_i + \tau_i K\}$, consider a homothet $\left(\sum \tau_i\right) K + o$ with center $o = \frac{\sum \tau_i o_i}{\sum \tau_i}$. If $\left(\sum \tau_i\right) K + o$ does not cover~$\mathcal{K}$, then there exists a hyperplane $H$ separating a point $p \in \conv \bigcup \mathcal{K} \setminus \left(\left(\sum \tau_i\right) K + o\right)$ from $\left(\left(\sum \tau_i\right) K + o\right)$. Projection onto the direction orthogonal to $H$ reveals a contradiction with Lemma~\ref{lem:segm}.

Another interesting approach to Goodmans' theorem was introduced by K.~Bezdek and A.~Litvak~\cite{bezdek2015packing}. They put the problem in the context of studying the packing analogue of Bang's problem through the LP-duality, which gives yet another proof of Goodmans' theorem for the case when $K$ is a Euclidean disk in the plane. One can adapt their argument for the original Bang's problem to get a ``dual'' counterpart of Goodmans' theorem. We discuss this counterpart and give our proof of a slightly more general statement in Section~\ref{sect:anti}.
% Unfortunately, this works only for round disks in the plane (and after an insignificant generalization in $\mathbb{R}^3$).

The paper is organized as follows.

In Section~\ref{sect:nonsymm} we prove a strengthening (with factor $\frac{d+1}{2}$ instead of $d$) of the following result of K.~Bezdek and Z.~Langi:
\begin{theorem}[K.~Bezdek and Z.~Langi, 2016]
\label{thm:nonsymmweak}
Given a non-separable family of positive homothetic copies of a (not necessarily centrally-symmetric) convex body $K \subset \mathbb{R}^d$ with homothety coefficients $\tau_1, \ldots, \tau_n > 0$, it is always possible to cover them by a translate of $d\left(\sum \tau_i\right)K$.
\end{theorem}

In Section~\ref{sect:simplex} we show that if we weaken the condition of non-separability considering only $d+1$ directions of separating hyperplanes, then factor $\frac{d+1}{2}$ cannot be improved.

In Section~\ref{sect:anti} we prove a counterpart of Goodmans' theorem related to the notion somehow opposite to non-separability: Given a positive integer $k$ and a family of Euclidean balls of radii $r_1, \ldots, r_n$ in $\mathbb{R}^d$, it is always possible to inscribe a ball of radius $r = \left(\sum r_i\right)/k$ within their convex hull, provided every hyperplane intersects at most $k$ interiors of the balls.

\section{A Goodmans-type result for non-symmetric bodies}
\label{sect:nonsymm}

Let $K \subset \mathbb{R}^d$ be a (not necessarily centrally-symmetric) convex body containing the origin and let $K^\circ = \{p : \langle p,q \rangle \le 1 \; \forall q \in K\}$ (where $\langle\cdot,\cdot\rangle$ stands for the standard inner product) be its polar body. We define the following \emph{parameter of asymmetry}:
\[
\sigma = \min\limits_{q\in \inte K} \min \{\mu>0: (K-q) \subset -\mu (K-q)\}
\]

It is an easy exercise in convexity to establish that $\min \{\mu>0: (K-q) \subset -\mu (K-q)\} = \min \{\mu>0: (K-q)^\circ \subset -\mu (K-q)^\circ\}$. So an equivalent definition (which is more convenient for our purposes) is
\[
\sigma = \min\limits_{q\in \inte K} \min \{\mu>0: (K-q)^\circ \subset -\mu (K-q)^\circ\}.
\]
The value $\frac{1}{\sigma}$ is often referred to as \emph{Minkowski's measure of symmetry} of body~$K$ (see, e.g.,~\cite{grunbaum1963measures}).

%For any direction consider two support hyperplanes of this direction and determine the ratio between two distances from $q$ to the hyperplanes. Denote $\sigma_q$ its maximal possible value over all directions. Then
%Consider its polar $K^\circ = \{p : \langle p,q \rangle \le 1 \; \forall q \in K\}$ (where $\langle\cdot,\cdot\rangle$ stands for the standard inner product). Introduce the following parameter of asymmetry $\sigma = \min \{ \mu>0: K \subset -\mu K  \}$. Note that $\sigma \ge 1$, and $\sigma = 1$ $\Leftrightarrow$ $K^\circ$ is centrally symmetric (with respect to the origin) $\Leftrightarrow$ $K$ is centrally symmetric (with respect to the origin). The value $\frac{1}{\sigma}$ is often referred as \emph{Minkowski's measure of symmetry} of the body $K$ (see, e.g.,~\cite[p.~246]{klee1963convexity}).

\begin{theorem}
\label{thm:nonsymm}
Given a non-separable family of positive homothetic copies of (not necessarily centrally-symmetric) convex body $K \subset \mathbb{R}^d$ with homothety coefficients $\tau_1, \ldots, \tau_n > 0$ it is always possible to cover them by a translate of $\frac{\sigma+1}{2}\left(\sum \tau_i\right)K$. (Here $\sigma$ denotes the parameter of asymmetry of $K$, defined above.)
\end{theorem}

\begin{proof}
We start by shifting the origin so that $K^\circ \subset -\sigma K^\circ$.

For a family $\mathcal{K} = \{o_i + \tau_i K\}$, consider the homothet $\frac{\sigma+1}{2} \left(\sum \tau_i\right) K + o$ with center $o = \frac{\sum \tau_i o_i}{\sum \tau_i}$. Assume that $\frac{\sigma+1}{2}\left(\sum \tau_i\right) K + o$ does not cover $\mathcal{K}$, hence there exists a hyperplane $H$ (strictly) separating a point $p \in \conv \bigcup \mathcal{K} \setminus \left(\frac{\sigma+1}{2}\left(\sum \tau_i\right) K + o\right)$ from $\left(\frac{\sigma+1}{2}\left(\sum \tau_i\right) K + o\right)$. Consider the orthogonal projection $\pi$ along $H$ onto the direction orthogonal to $H$. Suppose the segment $\pi(K)$ is divided by the projection of the origin in the ratio $1:s$. Since $K^\circ \subset -\sigma K^\circ$, we may assume that $s \in [1, \sigma]$. Identify the image of $\pi$ with the coordinate line $\mathbb{R}$ and denote $I_i = [a_i, b_i] = \pi\left(o_i + \tau_i K\right)$, $c_i = \pi(o_i)$, $\ell_i = b_i - a_i$, $L = \sum \ell_i$ (see Figure~\ref{pic:nonsymm}). Note that the $\ell_i$ are proportional to the $\tau_i$, and that $s(c_i - a_i) = b_i - c_i$. Denote $c = \pi(o) = \frac{\sum \ell_i c_i}{L}$ and $I = [a,b] = \pi\left(\frac{\sigma+1}{2}\left(\sum \tau_i\right) K + o\right)$ the segment of length $\frac{\sigma+1}{2}L$ divided by $c$ in the ratio $1:s$.
%Without loss of generality we assume that $c_i - a_i \le b_i - c_i \le \sigma(c_i - a_i)$ for all $i = 1, \ldots, n$.

\begin{figure}
\centering
\includegraphics{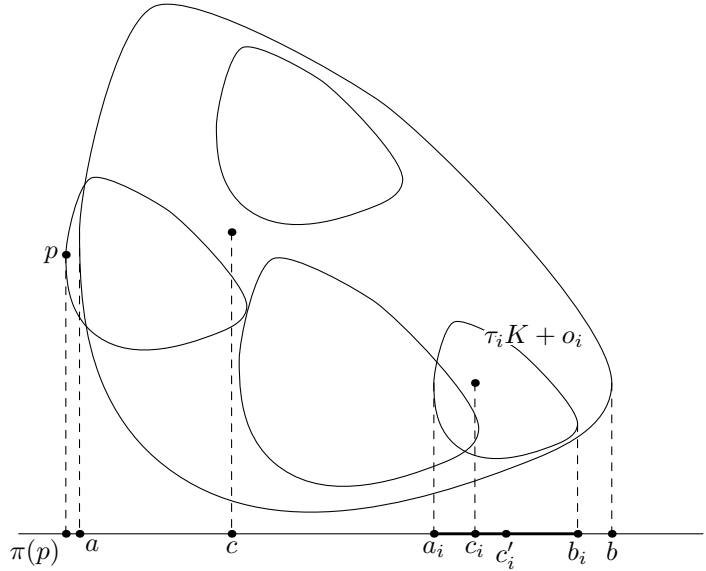}
\caption{Illustration of the proof of Theorem~\ref{thm:nonsymm}}
\label{pic:nonsymm}
\end{figure}

Also consider the midpoints $c_i' = \frac{a_i+b_i}{2}$. By Lemma~\ref{lem:segm}, the segment $I' = [a',b']$ of length $L$ with midpoint at $c' = \frac{\sum \ell_i c_i'}{L}$ covers the union $\bigcup I_i =\pi(\mathcal{K})$. Let us check that $I' \subset I$, which would be a contradiction, since $\pi(p) \in I'$, $\pi(p) \notin I$.

First, notice that $\displaystyle c_i' = \frac{a_i + b_i}{2} \ge \frac{s a_i + b_i}{1+s} = c_i$, hence
\[
a' = c' - \frac12 L \ge c - \frac12 L \ge c - \frac{1}{1+s} \frac{\sigma+1}{2} L = a.
\]

Second, $\displaystyle c_i' - c_i = \frac{a_i+b_i}{2} - \frac{sa_i + b_i}{1+s} = \frac{s-1}{s+1} \frac{\ell_i}{2}$, hence
\begin{multline*}
b' = c' + \frac12 L = c + \left(c'-c\right) + \frac12 L = c + \frac{s-1}{2\left(s+1\right)} \frac{\sum \ell_i^2}{L} + \frac12 L \le \\
\le  c + \frac{s-1}{2\left(s+1\right)} L + \frac12 L \le c + \frac{s}{1+s} \frac{\sigma+1}{2} L = b.
\end{multline*}
\end{proof}

\begin{lemma}[H.~Minkowski, J.~Radon]%[Gr\"unbaum~\cite{grunbaum1960partitions}, Hammer~\cite{hammer1960volumes}]
\label{lem:asymm}
Let $K$ be a convex body in $\mathbb{R}^d$.Then $\sigma \le d$, where $\sigma$ denotes the parameter of asymmetry of $K$, defined above.
\end{lemma}

For the sake of completeness we provide a proof here.

\begin{proof}

Suppose the origin coincides with the center of mass $g = \int\limits_K x \; dx / \int\limits_K dx$. We show that $K^\circ \subset -d K^\circ$.

Consider two parallel support hyperplanes orthogonal to one of the coordinate axes $Ox_1$. We use the notation $H_t = \{x = (x_1,\ldots,x_d): x_1 = t\}$ for hypeplanes orthogonal to this axis. Without loss of generality, these support hyperplanes are $H_{-1}$ and $H_s$ for some $s \ge 1$. We need to prove $s \le d$.
%Denote $A(t) = \volu_{d-1} (\tilde K \cap H_t)$, $t \in [-1,s]$, the $(d-1)$-dimensional volume of section $K \cap H_t$. Instead of $K$, consider its \emph{Schwarz symmetrization}, i.e. the body $\tilde K$ whose section $\tilde K \cap H_t$ is a ball of $(d-1)$-dimensional volume $A(t)$ with the center on the axis. The center of mass of $\tilde K$ is still at the origin. It follows from the Brunn--Minkowski inequality that the function $A(t)^{\frac{1}{d-1}}$ is concave on $[-1,s]$.

Assume that $s > d$. Consider a cone $C$ defined as follows: its vertex is chosen arbitrarily from $K \cap H_s$; its section $C \cap H_0 = K \cap H_0$; the cone is truncated by $H_{-1}$.
Since $C$ is a $d$-dimensional cone, the $x_1$-coordinate of its center of mass divides the segment $[-1, s]$ in ratio $1:d$. Therefore, the center of mass has positive $x_1$-coordinate.
It follows from convexity of $K$ that $C \setminus K$ lies (non-strictly) between $H_{-1}$ and $H_0$, hence the center of mass of $C \setminus K$ has non-positive $x_1$-coordinate. Similarly, $K \setminus C$ lies (non-strictly) between $H_0$ and~$H_s$, hence its center of mass has non-negative $x_1$-coordinate. Thus, the center of mass of $K = (C \setminus (C \setminus K)) \cup (K \setminus C)$ (see Figure~\ref{pic:cone}) must have positive $x_1$-coordinate, which is a contradiction.

\begin{figure}[h]
\centering
\includegraphics{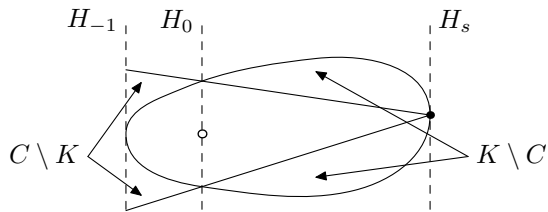}
\caption{Illustration of the proof of Lemma~\ref{lem:asymm}}
\label{pic:cone}
\end{figure}

%So, we can conclude that center of mass of $C \cap K$ has the positive $x_1$-coordinate.
%Since $K \setminus C$ lies (non-strictly) between $H_0$ and $H_s$ its center of mass has non-negative $x$-coordinate.
%We come to contradiction, because we obtain that $x$-coordinate of the center of mass of $K$ as the union of $C \cap K$ and $K \setminus C$ is strictly greater than $0$.

% On the other hand, the ratio $(r-(-1)):(s-r)$ equals $1:d$, as follows from the calculation:
% $$
% r = \frac{\int\limits_{-1}^{s} x \volu_{d-1}(C \cap H_x) \; dx}{\int\limits_{-1}^{s} \volu_{d-1}(C \cap H_x) \; dx} = \frac{\int\limits_{-1}^{s} x (s-x)^{d-1} \; dx}{\int\limits_{-1}^{s} (s-x)^{d-1} \; dx} = \frac{s-d}{d+1}.
% $$
% Therefore, $s = r(d+1) + d \le d$.
%
%$$
%\int\limits_C x \; dx = \int\limits_{-1}^s t \tilde A(t) \; dt \le \int\limits_{-1}^s t A(t) \; dt = 0,
%$$
%
%\begin{multline*}
%0 \ge \int\limits_{-1}^s t (s - t)^{d-1} \; dt = s \int\limits_{-1}^s (s - t)^{d-1} \; dt - \int\limits_{-1}^s (s - t)^d \; dt = \\
%= \frac{s (s+1)^d}{d} - \frac{(s+1)^{d+1}}{d+1} = \frac{(s+1)^{d+1}}{d} \left(\frac{s}{s+1} - \frac{d}{d+1}\right),
%\end{multline*}
%
%and $s \le d$.

\end{proof}

\begin{corollary}
\label{cor:nonsymmstrong}
The factor $d$ in Theorem~\ref{thm:nonsymmweak} can be improved to $\frac{d+1}{2}$.
\end{corollary}

\begin{proof}
The result follows from Theorem~\ref{thm:nonsymm} and Lemma~\ref{lem:asymm}.

An alternative proof of this corollary that avoids Lemma~\ref{lem:asymm} is as follows. We use the notation of Theorem~\ref{thm:nonsymmweak}. Consider the smallest homothet $\tau K$, $\tau > 0$, that can cover $\mathcal{K}$ (after a translation to $\tau K + t$, $t \in \mathbb{R}^d$). Since it is the smallest, its boundary touches $\partial \conv \bigcup \mathcal{K}$ at some points $q_0$, $\ldots$, $q_m$ ($m \le d$) such that the corresponding support hyperplanes $H_0$, $\ldots$, $H_m$ bound a \emph{nearly bounded} set $S$, i.e., a set that can be placed between two parallel hyperplanes.

Circumscribe all the bodies from the family $\mathcal{K}$ by the smallest homothets of $S$ and apply Theorem~\ref{thm:nonsymm} for them (note that if $m<d$ then $S$ is unbounded, but that does not ruin our argument). Since $S$ is a cylinder based on an $m$-dimensional simplex, its parameter of asymmetry equals $m \le d$, and we are done.

%common outer normals $n_0$, $\ldots$, $n_m$ surround the origin (i.e. $0 \in \conv \{n_0, \ldots, n_m\}$).
%Consider the relative location of the center of mass of simplex $\triangle = S \cap \aff\{q_0, \ldots, q_m\}$ with respect to $\tau K + t$ and shift the origin to place it in the same relative location with respect to $K$. Now we consider the family $\mathcal{K} = \{o_i + \tau_i K\}$ (with notation $o_i$ being changed), and claim that a homothet $\frac{d+1}{2} \left(\sum \tau_i\right) K + o$, where $o = \frac{\sum \tau_i o_i}{\sum \tau_i}$, covers $\mathcal{K}$. %Assume the contrary. Then

%Assuming the contrary, we find a hyperplane $H$ (strictly) separating a point $p \in \conv \bigcup \mathcal{K} \setminus \left(\frac{d+1}{2}\left(\sum \tau_i\right) K + o\right)$ from $\left(\frac{d+1}{2}\left(\sum \tau_i\right) K + o\right)$. Note that $H$ can be chosen parallel to a facet of $S$.%spanned by the facets of $\left(\frac{d+1}{2}\left(\sum \tau_i\right) K + o\right)$, so $H$ is parallel to one of them.

%After projecting everything along $H$ onto the direction, orthogonal to $H$, we introduce $a$, $a'$, $c$, $c'$, $L$, $s$ as in the proof of Theorem~\ref{thm:nonsymm}. Set $\sigma = d$ (note that $s \le d$) and repeat the same argument as before. So we show in the same manner that an inequality
%\[
%a' = c' - \frac12 L \ge c - \frac12 L = a
%\]
%holds contradicting our assumption.

\end{proof}

\begin{remark}
\label{rem:nonsymmbest}
Up to this moment the best possible factor for non-symmetric case is unknown. Bezdek and L\'angi~\cite{bezdek2016non} give a sequence of examples in $\mathbb{R}^d$ showing that it is impossible to obtain a factor less than $\frac23 + \frac{2}{3\sqrt{3}}$ $(> 1)$ for any $d \ge 2$.
\end{remark}

\section{A sharp Goodmans-type result for simplices}
\label{sect:simplex}

Consider the case when $K \subset \mathbb{R}^d$ is a simplex.
% Define \emph{allowed} directions as directions orthogonal to its facets.
In this section we are only interested in separating hyperplanes parallel to a facet of $K$.

\begin{theorem}
\label{thm:simplex}
Let $\mathcal{K}$ be a family of positive homothetic copies of a simplex $K \subset \mathbb{R}^d$ with homothety coefficients $\tau_1$, $\ldots$, $\tau_n > 0$. Suppose any hyperplane $H$ (parallel to a facet of $K$) intersecting $\conv \bigcup \mathcal{K}$ intersects a member of $\mathcal{K}$.
Then it is possible to cover $\bigcup \mathcal{K}$ by a translate of $\frac{d+1}{2}\left(\sum \tau_i\right)K$.
Moreover, factor $\frac{d+1}{2}$ cannot be improved.
\end{theorem}

\begin{proof}
A proof of possibility to cover follows the same lines as (and is even simpler than) the proof of Theorem~\ref{thm:nonsymm}. Let $K$ have its center of mass at the origin. For a family $\mathcal{K} = \{o_i + \tau_i K\}$, consider a homothet $\frac{d+1}{2} \left(\sum \tau_i\right) K + o$ with center $o = \frac{\sum \tau_i o_i}{\sum \tau_i}$.
Assuming $\frac{d+1}{2}\left(\sum \tau_i\right) K + o$ does not cover $\mathcal{K}$, we find a hyperplane $H$ (strictly) separating a point $p \in \conv \bigcup \mathcal{K} \setminus \left(\frac{d+1}{2}\left(\sum \tau_i\right) K + o\right)$ from $\left(\frac{d+1}{2}\left(\sum \tau_i\right) K + o\right)$.
Note that $H$ can be chosen among the hyperplanes spanned by the facets of $\left(\frac{d+1}{2}\left(\sum \tau_i\right) K + o\right)$, so $H$ is parallel to one of them.

After projecting everything along $H$ onto the direction orthogonal to $H$, we repeat the same argument as before and show that (in the notation from Theorem~\ref{thm:nonsymm})
\[
a' = c' - \frac12 L \ge c - \frac12 L = a,
\]
which contradicts our assumption.

Next, we construct an example showing that factor $\frac{d+1}{2}$ cannot be improved.

Consider a simplex \[K = \{x = (x_1, \ldots, x_d) \in \mathbb{R}^d: x_i \ge 0,\,\,\, \sum\limits_{i=1}^d x_i \le \frac{d(d+1)}{2} N + 1\},\] where $N$ is an arbitrary large integer. Section it with all hyperplanes of the form $\{x_i = t\}$ or of the form $\sum\limits_{i=1}^d x_i = t$ (for $t \in \mathbb{Z}$). Consider all the smallest simplices generated by these cuts and positively homothetic to $K$. We use coordinates
\[
\begin{pmatrix} b_1 \\ b_2 \\ \vdots \\ b_n \end{pmatrix}, \quad 0 \le b_i \in \mathbb{Z}, \quad \sum\limits_{i=1}^d b_i \le \frac{d(d+1)}{2} N,
\]
to denote the simplex lying in the hypercube $\{b_i \le x_i \le b_i+1, i = 1, \ldots, d\}$.

For $d=2$ (see Figure~\ref{pic:triangle}) we compose $\mathcal{K}$ of the simplices with the following coordinates:
\[
\begin{pmatrix}
0 \\
N
\end{pmatrix},
\begin{pmatrix}
1 \\
N+1
\end{pmatrix}, \ldots,
\begin{pmatrix}
N \\
2N
\end{pmatrix},
\begin{pmatrix}
N+1 \\
0
\end{pmatrix}, \ldots,
\begin{pmatrix}
2N \\
N-1
\end{pmatrix}.
\]

\begin{figure}
\centering
\includegraphics{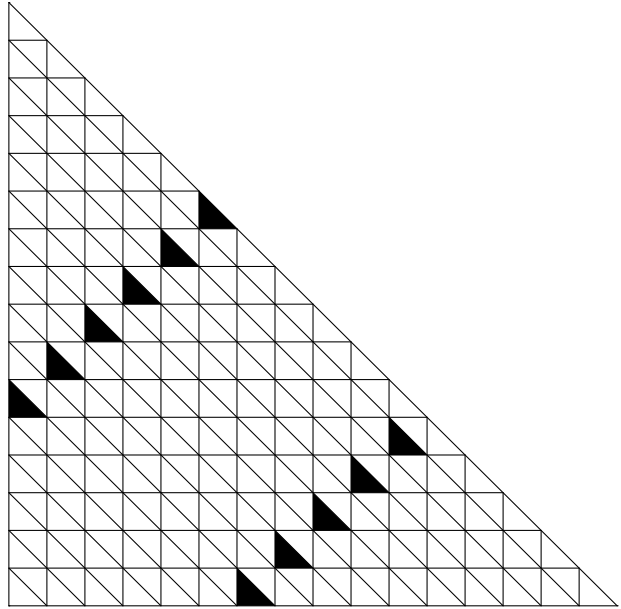}
\caption{Example for $d=2$ and $N=5$}
\label{pic:triangle}
\end{figure}

For $d=3$:
\[
\begin{pmatrix}
0 \\
N \\
2N
\end{pmatrix},
\begin{pmatrix}
1 \\
N+1 \\
2N+1
\end{pmatrix}, \ldots,
\begin{pmatrix}
N \\
2N \\
3N
\end{pmatrix},
\begin{pmatrix}
N+1 \\
2N+1 \\
0
\end{pmatrix}, \ldots,
\begin{pmatrix}
2N \\
3N \\
N-1
\end{pmatrix},
\begin{pmatrix}
2N+1 \\
0 \\
N
\end{pmatrix}, \ldots,
\begin{pmatrix}
3N \\
N-1 \\
2N-1
\end{pmatrix}.
\]

For general $d$:
\[
\begin{pmatrix}
0 \\
N \\
2N \\
\vdots \\
(d-1)N
\end{pmatrix},
\begin{pmatrix}
1 \\
N+1 \\
2N+1 \\
\vdots \\
(d-1)N+1
\end{pmatrix}, \ldots,
\begin{pmatrix}
i \pmod{dN+1} \\
N + i \pmod{dN+1}\\
2N + i \pmod{dN+1}\\
\vdots \\
(d-1)N + i \pmod{dN+1}
\end{pmatrix}, \ldots,
\begin{pmatrix}
dN \\
N-1 \\
2N-1 \\
\vdots \\
(d-1)N-1
\end{pmatrix}.
\]

It is rather straightforward to check that each $b_i$ ranges over the set $\{0, 1, \ldots, dN\}$, and their sum is not greater than $\frac{d(d+1)}{2} N$. Therefore, the chosen family $\mathcal{K}$ is indeed non-separable by hyperplanes parallel to the facets of $K$. Moreover, the chosen simplices touch all the facets of $K$, so $K$ is the smallest simplex covering $\mathcal{K}$. Finally, we note that any one-dimensional parameter of $K$ (say, its diameter) is $\frac{{d(d+1)} N}{{2}(dN+1)}$ times greater than the sum of the corresponding parameters of the elements of $\mathcal{K}$, and this ratio tends to $\frac{d+1}{2}$ as $N \to \infty$.

\end{proof}

\section{A ``dual'' version of Goodmans' theorem}
\label{sect:anti}

\begin{lemma}
\label{lem:antisegm}
Let $I_1, \ldots, I_n \subset \mathbb{R}$ be segments of lengths $\ell_1, \ldots, \ell_n$ with midpoints $c_1, \ldots, c_n$. Assume every point on the line belongs to at most $k$ of the interiors of the $I_i$. Then the segment $I$ of length $\frac{1}{k}\sum \ell_i$ with midpoint at the center of mass $c = \frac{\sum \ell_i c_i}{\sum \ell_i}$ lies in $\conv \bigcup I_i$.
\end{lemma}

\begin{proof}
Mark all the segment endpoints and subdivide all the segments by the marked points. Next, put the origin at the leftmost marked point and numerate the segments between the marked points from left to right. We say that the $i$-th segment is of multiplicity $0 \le k_i \le k$ if it is covered $k_i$ times. We keep the notation $I_i$ for the new segments with multiplicities, $c_i$ for their midpoints, and~$\ell_i$ for their lengths. Note that the value $\frac{\sum \ell_i c_i}{\sum \ell_i}$ is preserved after this change of notation: it is the coordinate of the center of mass of the segments regarded as solid one-dimensional bodies of uniform density.

Note that $c_i = \ell_1 + \ldots + \ell_{i-1} + \frac12 \ell_i$. We prove that $c = \frac {\sum k_i\ell_i c_i}{\sum k_i\ell_i} \ge \frac{\sum k_i \ell_i}{2k}$ (this would mean that the left endpoint of $I$ is contained in $\conv \bigcup I_i$; for the right endpoint everything is similar).

The inequality in question
\[
2c \sum_i k_i \ell_i = k_1 \ell_1 \cdot \ell_1 + k_2 \ell_2 \cdot \left(2\ell_1 + \ell_2\right) + k_2 \ell_2 \cdot \left(2\ell_1 + 2\ell_2 + \ell_3\right) + \ldots \stackrel{?} \ge \frac{1}{k} \left( \sum_i k_i\ell_i \right)^2
\]
is equivalent to
\[
k \left( \sum_i k_i\ell_i^2 + 2 \sum_{i<j} k_j\ell_i\ell_j \right)\stackrel{?} \ge \left( \sum_i k_i\ell_i \right)^2,
\]
which is true, since $k \ge k_i$.

\end{proof}

%\begin{theorem}
%\label{thm:anti}
%Given a family $\mathcal{K}$ of positive homothetic copies of a centrally-symmetric convex body $K \subset \mathbb{R}^d$ with homothety coefficients $\tau_1$, $\ldots$, $\tau_n > 0$ and with property that every hyperplane intersects at most $k$ interiors of the homothets.
%Then $\mathcal{K}$ can be covered by some translate of the body $\frac{1}{k}\left(\sum \tau_i\right) K$.
%\end{theorem}

\begin{theorem}
\label{thm:anti}
Let $k$ be a positive integer, and $\mathcal{K}$ be a family of positive homothetic copies (with homothety coefficients $\tau_1, \ldots, \tau_n > 0$) of a centrally-symmetric convex body $K \subset \mathbb{R}^d$. Suppose any hyperplane intersects at most $k$ interiors of the homothets. Then it is possible to put a translate of $\frac{1}{k}\left(\sum \tau_i\right) K$ into their convex hull.
\end{theorem}

\begin{proof}
As usual, for a family $\mathcal{K} = \{o_i + \tau_i K\}$, consider a homothet \mbox{$\frac{1}{k}\left(\sum \tau_i\right) K + o$} with center $o = \frac{\sum \tau_i o_i}{\sum \tau_i}$. Assume $\frac{1}{k} \left(\sum \tau_i\right) K + o$ does not fit into $\conv \bigcup \mathcal{K}$, then there exists a hyperplane $H$ separating a point $p \in \frac{1}{k} \left(\sum \tau_i\right) K + o$ from $\conv \bigcup \mathcal{K}$. After projecting onto the direction orthogonal to $H$, we use Lemma~\ref{lem:antisegm} to obtain a contradiction.
\end{proof}

\begin{remark}
The estimate in Theorem~\ref{thm:anti} is sharp for any $k$, as can be seen from the example of $k$ translates of $K$ lying along the line so that consecutive translates touch.
\end{remark}

\subsection*{Acknowledgements.}
The authors are grateful to Rom Pinchasi and Alexandr Polyanskii for fruitful discussions. Also the authors thank Roman Karasev, Kevin Kaczorowski, and the anonymous referees for careful reading and suggested revisions.

The research of the first author is supported by People Programme (Marie Curie Actions) of the European Union's Seventh Framework Programme (FP7/2007-2013) under REA grant agreement n$^\circ$[291734].
The research of the second author is supported by he Russian Foundation for Basic Research grant 15-01-99563 A and  grant 15-31-20403 (mol\_a\_ved).

% BibTeX users please use one of
% \bibliographystyle{spbasic}      % basic style, author-year citations
% \bibliographystyle{spmpsci}      % mathematics and physical sciences
%\bibliographystyle{spphys}       % APS-like style for physics
%\bibliography{}   % name your BibTeX data base

% Non-BibTeX users please use
% 
% 
% \begin{thebibliography}{}
% %
% % and use \bibitem to create references. Consult the Instructions
% % for authors for reference list style.
% %
% \bibitem{RefJ}
% % Format for Journal Reference
% Author, Article title, Journal, Volume, page numbers (year)
% % Format for books
% \bibitem{RefB}
% Author, Book title, page numbers. Publisher, place (year)
% % etc
% \end{thebibliography}
% \bibliography{balitskiy}

\vskip 1cm

\end{document}